\documentclass[12pt]{amsart}
\usepackage[margin=1.25in]{geometry}
\usepackage{graphicx}
\usepackage{amssymb}
\usepackage{enumerate}
\usepackage{epstopdf}
\usepackage{amsmath,amsthm,amscd,amssymb}
\usepackage{latexsym}
\usepackage[colorlinks,citecolor=red,pagebackref,hypertexnames=false]{hyperref}
\usepackage{graphicx}
\usepackage[font=small,labelfont=bf]{caption}
\usepackage{mathtools}
\numberwithin{equation}{section}
\newtheorem{theorem}{Theorem}[section]
\newtheorem{corollary}[theorem]{Corollary}

\newtheorem{lemma}[theorem]{Lemma}

\newtheorem{example}{Example}
\newtheorem{definition}[theorem]{Definition}

\DeclareMathAlphabet{\mymathbb}{U}{BOONDOX-ds}{m}{n}

\begin{document}

\title[Unimodality of Bernoulli Random Quota Complexes]{Unimodality of the Expectation of Betti Numbers for Bernoulli Random Quota Complexes}

\author[E. Crossen Brown]{Erin Crossen Brown}
\address{E. Crossen Brown, Department of Mathematics, 
University of Rochester,
500 Joseph C. Wilson Blvd., Rochester, NY 14627}
\email{ecrossen@ur.rochester.edu}



\date{}


\begin{abstract}
We study certain random simplicial complexes, called random quota complexes. A quota complex on $N+1$ weighted vertices is constructed by adding an $n$-simplex if the sum of the weights of the vertices is below a given quota, $q$. In this paper, the weights of the vertices are chosen i.i.d. with a Bernoulli distribution. The main result of this paper is that the expectation of the $m^{\textrm{th}}$ Betti number, i.e., the dimension of the $m^{\textrm{th}}$ homology group, is unimodal in $m$. 

\noindent
{\it Keywords: Random graphs, random simplicial complexes, random quota complexes.

}


\end{abstract}

\maketitle

\setcounter{tocdepth}{1}
\tableofcontents
\setcounter{tocdepth}{3}

\section{Introduction}

In this paper, we study the expected homology of a certain family of random quota complexes, where the vertices are assigned random weights by a Bernoulli process. Random quota complexes are a specific family of random simplicial complexes, which are the higher dimensional analogs of random graphs. 


Random graph theory has been studied since the 1940s and 1950s, beginning with some results of Paul Erd\H{o}s and Alfr\'{e}d R\'{e}nyi. They used a probabilistic method approach to prove the existence of graphs with certain properties that had been in question and established facts about Ramsey numbers \cite{bollobas}. Random graph theory has become a topic of major importance in Computer Science, where it is referred to as network theory. In network theory, random graphs are used to represent network flow, where each edge has a capacity and a direction of flow; the shortest path problem; circuit path analysis, where a node can fail randomly at any time; social networks; and more \cite{circuit, lusher, newman}. In physics, percolation theory is another area of study that uses random graph theory. Here, the question is, if liquid is poured through a porous material, will the liquid make it to the bottom? The holes in the material are the vertices of the graph and the paths that the liquid makes between the holes are the edges. This has applications in the study of the spread of epidemics, environment fragmentation and its effects on animal habitats, dental percolation, and more \cite{grimmett, roberts, vanderhofstad}. In all these different models, one can observe phase transitions in the different properties that the models possess, such as porous to not porous, connected to disconnected networks, etc., as parameters vary.  

A particular random graph of interest is the Erd\H{o}s-R\'{e}nyi model, a random graph $G(n,p)$ that has vertex set of size $n$, where every possible edge is placed in the graph independently with probability $p$. Erd\H{o}s and R\'{e}nyi showed in 1960 that with $\epsilon>0$ fixed, if $p\geq \frac{(1+\epsilon)\log n}{n}$, then $G$ is connected with high probability, and if $p\leq \frac{(1-\epsilon)\log n}{n}$, then $G$ is disconnected with high probability \cite{erdos1960}. 

More recently, many authors have begun to study random simplicial complexes, which are an extension of random graphs to higher dimensions. 
This survey \cite{kahle} by Kahle lists many authors' new findings in Random Simplicial Complexes, such as \cite{kahlemeckes}, \cite{babson}, and \cite{aronshtam}. Often, authors study phase transitions in the expectation of the features of these random simplicial complexes as a function of the parameters. In particular, in \cite{linial2006}, a phase transition in the first Betti number was shown for a particular family of 2-dimensional random simplicial complexes. For this extension, faces are added to the complete graph on $n$ vertices independently with probability $p$. It was shown that, for a function $\omega(n)$ that tends to infinity, if $p\leq \frac{2\log n - \omega(n)}{n}$, then the first 
Betti number is 0
with high probability, but if $p\geq \frac{2\log n + \omega(n)}{n}$, then the first Betti number is nonzero
with high probability. This was later extended further to random $k$-dimensional complexes in \cite{meshulam}. In general, the study of phase transitions in random simplicial complexes has generated a lot of interest. 

In this paper, a specific type of random simplicial complex, called a random quota complex, is studied. We assign each vertex a weight independently as  a Bernoulli random variable + 1 so that the possible weights are 1 or 2. We put in faces when the sum of the weights of the vertices is below quota. In section~\ref{sec: formula}, we give a formula for the expectation of the reduced $m^{th}$ Betti number, $\mathbb{E}[\widetilde{\beta}_{m}]$.
We next show that $\mathbb{E}[\widetilde{\beta}_{m}]$ is unimodal in $m$ under certain restrictions, as in the following theorem [for definitions, see section~\ref{sec: def}]: 

\begin{theorem}
\label{thm: unimodal}
Fix $d>1/2$, where $N+1=dq$ and consider a random Bernoulli quota complex with $N+1$ vertices, quota $q$, and Bernoulli probability $p$.
Then $\mathbb{E}[\widetilde{\beta}_{m}]$ is unimodal in $m$. 
\end{theorem}

Theorem~\ref{thm: unimodal} will be proved in section~\ref{sec: proof_unimodal}. We also show that the peak is unique and give equations to find the peak in section~\ref{sec: peak}. Lastly in section~\ref{sec: bounds}, using bounds on the Betti numbers, we determine the fate of the Betti numbers to establish regimes where they grow or die off.


\section{Definitions and preliminary results}
\label{sec: def}

We direct the reader to Munkres' Elements of Algebraic Topology \cite{munkres} for basic definitions concerning simplicial complexes and homology. The following definitions are also used throughout this paper. 

Associated to a topological space, there are numbers which measure the various dimensional holes in the space, which are called Betti numbers. These are defined below. 


\begin{definition}
(Betti number) Using coefficients in $\mathbb{R}$, the $m^{\textrm{th}}$ Betti number of the complex $X$, $\beta_{m}$, is the dimension of the $m^{\textrm{th}}$ homology group of $X$, i.e., $\beta_{m}=dim(H_m(X,\mathbb{R}))$. 
The reduced $m^{\textrm{th}}$ Betti number of the complex $X$, $\widetilde{\beta}_{m}$, is the dimension of the reduced $m^{\textrm{th}}$ homology group of $X$, i.e., $\widetilde{\beta}_{m}=\textrm{dim}(\widetilde{H_m}(X))$. 

In general, $\widetilde{\beta}_{m}=\beta_{m}$ for all $m\geq1$, and $\widetilde{\beta}_{0}=\beta_0-1$. 
\end{definition} 

From now on, all references to Betti numbers will mean reduced Betti numbers.

\begin{definition}
(Scalar-valued quota complex). Let $V$ be a vertex set. A scalar-valued quota system on $V$ is given by a weight function $w:V\rightarrow \mathbb{R}_+$ and quota $q>0$. The quota complex $X[w:q]$ is the simplicial complex on the vertex set $V$ such that a face $F=[v_0,...,v_n]$ is in $X[w:q]$ if and only if $w(F)\coloneqq\sum_{i=0}^{n}w(v_i)<q$ \cite{PW}.
\end{definition}


\begin{example}
For example, consider a system of four vertices with weights $w(v_0)=1$, $w(v_1)=3$, $w(v_2)=4$, and $w(v_3)=7$, and let the quota $q=10$. Then the edges $[v_0, v_1]$, $[v_0, v_2]$, $[v_0, v_3]$, and $[v_1, v_2]$ are all included in the complex, though the edges $[v_1, v_3]$ and $[v_2, v_3]$ are not, because $3+7=10\geq10$ and $4+7=11\geq10$. Also, the triangle $[v_0, v_1, v_2]$ is included, but the triangle $[v_0, v_1, v_3]$ is not because $1+3+7=11\geq10$. If the quota was raised to $q=12$, then the edges $[v_1, v_3]$ and $[v_2, v_3]$ and the triangle $[v_0, v_1, v_3]$ would now also be included in addition to what was previously included, though the triangles $[v_0, v_2, v_3]$ and $[v_1, v_2, v_3]$ would still not be included. Note that the 3-simplex $[v_0,v_1,v_2,v_3]$ would not be included in this complex until the quota is raised to be strictly greater than 15. 
\end{example}

We use random variants of quota complexes throughout this paper, so we will define those next. 

\begin{definition}
(Random quota complex). Fix $X_0=m>0$ as a nonrandom value. Let $X_1$, ..., $X_N$ be independent random variables. Fix a quota $q>m$. Then, let $X=\{X_0=m, X_1, ..., X_N\}$. $X[q]$ is the quota complex on vertices $\{0,1,2,...,N\}$ with weights $w(i)=X_i$ and quota $q$.  

$X[q]$ is called a random scalar quota complex. Given a realization of the random weights $X_1,\dots,X_N$,
$X[q]$ will determine a specific scalar quota complex. Thus $X[q]$ can be viewed as a random variable 
which takes values in the set of finite abstract simplicial complexes  \cite{PW}. 
\end{definition}

In \cite{PW}, the following theorem regarding the relationship between the $m^{\textrm{th}}$ Betti number of a quota complex and the quota itself was proven. 

\begin{theorem}
\label{thm: Jon}
Let $X=X[w : q]$ be a scalar valued quota complex, then $X$ is homotopy equivalent to a bouquet of spheres. Let $v_0$ be a vertex of minimum weight. Then there is one sphere of dimension $s$ in the bouquet for every face $F$ of dimension $s$ in $X$, not containing $v_0$, such that $q-w(v_0)\leq w(F)< q$, where $w(v_0)$ is the smallest weight. 

\end{theorem}







\subsection{The Bernoulli quota complex model}
Let $X_1, \dots, X_N$ be independent, identically distributed, Bernoulli random variables with probability $p$ to have value 1 and probability $\overline{p}\coloneqq(1-p)$ to have value 0 (note that $\overline{p}$ will be used as $(1-p)$ throughout this paper). 

Generate a random quota system, $X$ with quota $q$ on $N+1$ vertices $v_0,\dots,v_N$ with weights $W_0,...,W_N$, respectively, where 
\begin{align*}
    W_0 &= 1 &\textrm{is fixed,} \\
    W_j &= 1 + X_j,  &j=1,\dots,N.
\end{align*}

This means that the possible weights are 1 or 2. 

According to Theorem~\ref{thm: Jon}, the reduced $m^{\textrm{th}}$ Betti number of this complex, $\widetilde{\beta}_{m}$ is the number of $m+1$ sets of vertices, which do not include $v_0$, whose total weight lies in the range

\[
[q-W_0,q)=[q-1,q).
\]

We will consider when the quota is an integer, so in this case, we are looking for sets of vertices whose weights add to exactly $q-1$, since the weights we consider are integers.


\section{Expectation of the $m^{\textrm{th}}$ Betti number}
\label{sec: formula}

In this section, we find the formula for the expectation of the $m^\textrm{th}$ Betti number. This will be used throughout the rest of the paper.

\begin{lemma}
\label{lemma: exval}
If $\frac{q-1}{2}\leq m+1<q$, then  $$\mathbb{E}[\widetilde{\beta}_{m}]=\binom{N}{m+1}\binom{m+1}{q-m-2}p^{q-m-2}\;\overline{p}^{2m+3-q}.$$
If otherwise, then $\widetilde{\beta}_{m}=0$.  
\end{lemma}

\medskip

\begin{proof}

The $m^{\textrm{th}}$ Betti number is the number of sets of $m+1$ vertices whose weights add to be within the range $[q-1,q)$. Since the only possible weight values are $1$ and $2$, the sum of the weights of $m+1$ vertices is in the range $[m+1,2(m+1)]$. Therefore, there is no possibility of having the sum of the weights of a set of $m+1$ vertices adding to be within the range $[q-1,q)$ if $q-1>2(m+1)$ or if $q\leq m+1$. Thus, $\widetilde{\beta}_{m}=0$ outside of $m+1<q\leq2(m+1)+1$, i.e. outside of $\frac{q-1}{2}\leq m+1<q$, and so we assume that $m$ is in this range for the rest of the proof. 


\begin{align*}
    \mathbb{P}&[W_{j_1}+\dots+W_{j_{m+1}}\in[q-1,q)]=\mathbb{P}[(X_{j_1}+1)+...+(X_{j_{m+1}}+1)\in[q-1,q)] \\
    &=\mathbb{P}[X_{j_1}+...+X_{j_{m+1}}\in[q-1-(m+1),q-(m+1))] \\
    &= \mathbb{P}[X_{j_1}+...+X_{j_{m+1}}=q-1-(m+1)],
\end{align*}
where the last equation holds since $X_{j_1}+...+X_{j_{m+1}}$ is an integer. 






\medskip

\noindent Since the sum of i.i.d. Bernoulli random variables is binomial, 
\begin{align*}
    \mathbb{P}&[q-m-2\textrm{ successes out of }m+1\textrm{ trials}]=\binom{m+1}{q-m-2}p^{q-m-2}\;\overline{p}^{m+1-(q-m-2)} \\
    &=\binom{m+1}{q-m-2}p^{q-m-2}\;\overline{p}^{2m+3-q}.
\end{align*}

\medskip


\noindent Therefore, the expected value of the $m^{\textrm{th}}$ Betti number is:  $$\mathbb{E}[\widetilde{\beta}_{m}]=\binom{N}{m+1}\binom{m+1}{q-m-2}p^{q-m-2}\;\overline{p}^{2m+3-q}.$$


\end{proof}

\section{Unimodality of the expectation of the $m^{\textrm{th}}$ Betti number}
\label{sec: proof_unimodal}

In this section, we prove Theorem~\ref{thm: unimodal}. In order to do this, we must show that the sequence of numbers $\{\mathbb{E}[\widetilde{\beta}_{m}]\}_{m=0}^\infty$ is unimodal in $m$. Recall that we let $d>1/2$ and $N+1=dq$.

In order for a sequence of nonnegative real numbers $\{f_n\}_{n=0}^\infty$ to be uniquely unimodal, the sequence must go from increasing to decreasing at some unique point. Define the first forward quotient $M^1(f_n)=f_{n}/f_{n-1}$ and the second forward quotient $M^2(f_n)=M^1(f_{n})/M^1(f_{n-1})$. Then the sequence $\{f_n\}_{n=1}^\infty$ is uniquely unimodal if $M^2(f_n)>1$ for all $n\geq 2$, and $M^1(f_{n_1})>1$ and $M^1(f_{n_2})<1$ for some $1\leq n_1<n_2$.


\medskip
\begin{proof}[Proof of Theorem~\ref{thm: unimodal}]

First, note that a more useful version of the expected value in Lemma~\ref{lemma: exval} is: 

\begin{align*}
    \mathbb{E}[\widetilde{\beta}_{m}]&=\binom{N}{m+1}\binom{m+1}{q-m-2}p^{q-m-2}\;\overline{p}^{2m+3-q} \\
    &=\dfrac{N!}{(N-m-1)!(q-m-2)!(2m+3-q)!}p^{q-m-2}\;\overline{p}^{2m+3-q}. 
\end{align*}


        
        
    
    
        
        





Next, we will find and simplify the forward quotient, $M^1(\mathbb{E}[\widetilde{\beta}_{m}])$:

\begin{align*}
    M^1(&\mathbb{E}[\widetilde{\beta}_{m}])=\frac{\mathbb{E}[\widetilde{\beta}_{m}]}{\mathbb{E}[\widetilde{\beta}_{m-1}]} \\
    &=\frac{\left[N!p^{q-m-2}\;\overline{p}^{2m+3-q}\right]\left[(N-(m-1)-1)!(q-(m-1)-2)!(2(m-1)+3-q)!\right]}{\left[(N-m-1)!(q-m-2)!(2m+3-q)!\right]\left[N!p^{q-(m-1)-2}\;\overline{p}^{2(m-1)+3-q}\right]} \\
    &=\frac{\overline{p}^2}{p}\dfrac{(N-m)(q-m-1)}{(2m+3-q)(2m+2-q)}.
\end{align*}

\bigskip

Next, we find the second forward quotient, $M^2(\mathbb{E}[\widetilde{\beta}_{m}])$: 

\begin{align*}
    M^2(\mathbb{E}[\widetilde{\beta}_{m}])&=\dfrac{\mathbb{E}[\widetilde{\beta}_{m}]/\mathbb{E}[\widetilde{\beta}_{m-1}]}{\mathbb{E}[\widetilde{\beta}_{m-1}]/\mathbb{E}[\widetilde{\beta}_{m-2}]} \\
    &=\dfrac{\overline{p}^2(N-m)(q-m-1)}{p(2m+3-q)(2m+2-q)} \dfrac{p(2(m-1)+3-q)(2(m-1)+2-q)}{\overline{p}^2(N-(m-1))(q-(m-1)-1)} \\
    &=\dfrac{(2m-q)}{(2m-q+2)}\dfrac{(2m-q+1)}{(2m-q+3)} \dfrac{(N-m)}{(N-m+1)}\dfrac{(q-m-1)}{(q-m)}.
\end{align*}

Considering each term in the final expression, it is easy to see that each denominator is at least one greater than its numerator. Therefore, this term must be less than 1 for all $m\geq2$, which implies that the second forward quotient is always less than 1. 





\medskip

Next, we show that the first forward quotient, $M^1(\mathbb{E}[\widetilde{\beta}_{m}])$, is equal to 1 at some point. To do this, we will look at the bounds for $m$ and consider the smallest and largest values that $m$ can take. The smallest value will make $M^1(\mathbb{E}[\widetilde{\beta}_{m}])>1$ and the largest will make $M^1(\mathbb{E}[\widetilde{\beta}_{m}])<1$.

From before, we know that $\mathbb{E}[\widetilde{\beta}_{m}]=0$ unless $\dfrac{q-1}{2}\leq m+1<q$. Also, since we have at most $N$ vertices that are not $v_0$, $m+1\leq N$. 

\underline{Smallest $m$}: For the first forward quotient, the smallest $m$ we can consider is $\dfrac{q+1}{2}$ (assume $q\geq1$ for $M^1(\mathbb{E}[\widetilde{\beta}_{m}])$ to be defined): 
\begin{align*}
    M^1(\mathbb{E}[\widetilde{\beta}_{\frac{q+1}{2}}])&= \frac{\overline{p}^2}{p}\frac{\left(N-\left(\frac{q+1}{2}\right)\right)\left(q-\left(\frac{q+1}{2}\right)-1\right)}{(2\left(\frac{q+1}{2}\right)+3-q)(2\left(\frac{q+1}{2}\right)+2-q)} \\
    &= \frac{\overline{p}^2}{p}\frac{\left(N-\frac{q}{2}-\frac{1}{2}\right)\left(\frac{q}{2}-\frac{3}{2}\right)}{(4)(3)}.
\end{align*}

Using the relation $N+1=dq$, we get:
  
\[
    M^1(\mathbb{E}[\widetilde{\beta}_{\frac{q+1}{2}}])=\frac{\overline{p}^2}{24p}\left(\left(d-\frac{1}{2}\right)q-\frac{3}{2}\right)\left(q-3\right).
\]

To show that this final expression is greater than 1 for sufficiently large $q$, consider term-by-term. The first term, $\dfrac{\overline{p}^2}{24p}$, is a constant with respect to $q$. Therefore, we can choose a large enough $q$ so that this term doesn't matter. The second term, $\left(q\left(d-\dfrac{1}{2}\right)-1\right)$ is greater than 1 if $q>\dfrac{5}{2d-1}$. The last term, $\left(q-3\right)$, is greater than 1 if $q>4$. 

Therefore, choose $q>\max\left\{4,\dfrac{5}{2d-1}\right\}$, or larger if necessary. Then $M^1(\mathbb{E}[\widetilde{\beta}_{m}])>1$ for the smallest possible value of $m$.

\underline{Largest $m$}: We actually have two different upper bounds for $m$, which are $m+1<q$ and $m+1\leq N$. We will consider each upper bound separately. 

First, consider the upper bound $m+1<q$. Assume that $q\in\mathbb{Z}$, so that our upper bound becomes $m=q-2$. Note that by choosing this upper bound, we must have that $q-2\leq N-1$. In the relation $N+1=dq$, this implies that $d\geq1$.

\begin{align*}
    M^1(\mathbb{E}[\widetilde{\beta}_{q-2}])&= \dfrac{\overline{p}^2}{p}\dfrac{(N-(q-2))(q-(q-2)-1)}{(2(q-2)+3-q)(2(q-2)+2-q)} \\
    &= \dfrac{\overline{p}^2}{p}\dfrac{(N-q+2)(1)}{(q-1)(q-2)}. 
\end{align*}

Using the relation $N+1=dq$, we get:
    
\[
    M^1(\mathbb{E}[\widetilde{\beta}_{q-2}])= \dfrac{\overline{p}^2}{p}\dfrac{(q(d-1)+1)}{(q-1)(q-2)}.
\]

As $q\rightarrow\infty$, this is certainly less than 1.


\medskip

Next, consider the upper bound $m=N-1$, which implies that $N-1\leq q-2$. This implies that $d\leq 1$. We have

\begin{align*}
    M^1(\mathbb{E}[\widetilde{\beta}_{N-1}])&= \dfrac{\overline{p}^2}{p}\dfrac{(N-(N-1))(q-(N-1)-1)}{(2(N-1)+3-q)(2(N-1)+2-q)} \\
    &= \dfrac{\overline{p}^2}{p}\dfrac{(1)(q-N)}{(2N-q+1)(2N-q)}.
\end{align*}

Using the relation $N+1=dq$, we get:
  
\[
     M^1(\mathbb{E}[\widetilde{\beta}_{N-1}])=\dfrac{\overline{p}^2}{p}\dfrac{(q(1-d)+1)}{(q(2d-1)-1)(q(2d-1)-2)}. 
\]

As $q\rightarrow\infty$, this is certainly less than 1.

Therefore, for the lower bound for $m$, the first forward quotient is greater than 1, and for the upper bounds for $m$, the first forward quotient is less than 1. Since the second forward quotient is always less than 1, the first forward quotient must be strictly decreasing, so it must at some point equal 1. 



\end{proof}

\section{Determining $m_{\textrm{peak}}$}
\label{sec: peak}

Now, we would like to find the location of the mode, $m_{\textrm{peak}}$, of the expectation of the $m^{\textrm{th}}$ Betti number in terms of the parameters of the model. To further this, we introduce $\tau=\frac{m+1}{q}$, the comparable location of the dimension $m$ to the quota $q$. Then $\tau_{\textrm{peak}}=\frac{m_{\textrm{peak}}+1}{q}$, and as we consider asymptotically large quotas, we will look at
$$\tau_\infty=\tau_{\textrm{peak}}^\infty=\lim_{q\rightarrow\infty}\tau_{\textrm{peak}}.$$

\begin{theorem}
\label{thm: tinf}


If $p\neq3-\sqrt{8}$, then the mode of $\mathbb{E}[\widetilde{\beta}_{m_{\textrm{peak}}}]$ is achieved when $\frac{1}{2}\leq\tau_{\infty}<1$ solves the quadratic equation: 
\begin{equation*}
    \tau_{\infty}^2(\overline{p}^2-4p) + \tau_{\infty}\left(4p-\overline{p}^2(d+1)\right) +\overline{p}^2d-p=0.
\end{equation*}

If $p=3-\sqrt{8}$, then the mode is achieved when $\tau_{\infty}=\dfrac{4d-1}{4d}$.
\end{theorem}






\begin{proof}

As the mode occurs at the value of $m$ that makes the first forward quotient equal to 1, we have:

\begin{align*}
    1= \dfrac{\mathbb{E}[\widetilde{\beta}_{m}]}{\mathbb{E}[\widetilde{\beta}_{m-1}]}=  \dfrac{\overline{p}^2}{p}\dfrac{(N-m_{\textrm{peak}})(q-m_{\textrm{peak}}-1)}{(2m_{\textrm{peak}}+3-q)(2m_{\textrm{peak}}+2-q)}.
\end{align*}

Using the relations $N+1=dq$ and $\tau_{\textrm{peak}}=\dfrac{m_{\textrm{peak}}+1}{q}$, we get:

\begin{align*}
    1&= \dfrac{\overline{p}^2}{p}\dfrac{(dq-(m_{\textrm{peak}}+1))(q-(m_{\textrm{peak}}+1))}{(2(m_{\textrm{peak}}+1)+1-q)(2(m_{\textrm{peak}}+1)-q)} \\
    &= \dfrac{\overline{p}^2}{p}\dfrac{(d-\tau_{\textrm{peak}})(1-\tau_{\textrm{peak}})q}{((2\tau_{\textrm{peak}} -1)q+1)(2\tau_{\textrm{peak}} -1)}. 
\end{align*}

Solving for $q$ here leads to: 

$$q=\dfrac{p(2\tau_{\textrm{peak}}-1)}{\overline{p}^2(d-\tau_{\textrm{peak}})(1-\tau_{\textrm{peak}})-p(2\tau_{\textrm{peak}}-1)^2}.$$

Since we let $N$ tend to infinity, $q$ must also got to infinity by the relation $N+1=dq$, so in the above expression, the denominator must tend to 0 (the case where $\tau=1/2$ is handled separately).

Note that $\tau_{\textrm{peak}}\in\left[\frac{1}{2},1\right]$, which is a compact interval. Thus, for every sequence of $q$ going to infinity, there is a corresponding sequence of $\tau_{\textrm{peak}}$ which must have a convergent subsequence with a limit by the Bolzano-Weierstrass theorem. We will show that setting the above denominator equal to zero results in a quadratic equation in $\tau_{\textrm{peak}}$. We will then show that only one of the roots is a feasible limit value, $\tau_\infty$. If the sequence of $\tau_{\textrm{peak}}$ itself did not converge to this limit, there would be a ball of radius $\varepsilon>0$ around $\tau_\infty$ that a subsequence stayed out of, which contradicts the fact that all subsequences go to the same limit. Thus, $\tau_{\textrm{peak}}$ must converge to the unique feasible root of the quadratic equation, $\tau_\infty$. 

\medskip

Set the denominator equal to zero, and note that it is a quadratic with respect to $\tau_{\infty}$, which we will call $T(\tau_\infty)$: 

$$T(\tau_\infty)\coloneqq\tau_{\infty}^2(\overline{p}^2-4p) + \tau_{\infty}\left(4p-\overline{p}^2(d+1)\right) +\overline{p}^2d-p=0.$$

First, consider the case when this becomes linear in $\tau_{\infty}$, i.e., when $\overline{p}^2=4p$. Solving for $p$, we have $p=3\pm\sqrt8$. Since $p$ is a probability, only consider $p=3-\sqrt8$. Plugging in $\overline{p}^2=4p$, we have: 

\begin{align*}
    \tau_{\infty} &= \dfrac{p-\overline{p}^2d}{4p-\overline{p}^2(d+1)} \\
    &= \dfrac{4d-1}{4d}.
\end{align*}

The smallest value of $d$ is $\frac12$. Using this, we have that $\tau_{\infty}=\frac12$, which is the minimum value of $\tau$. As $d$ grows, $\tau_{\infty}$ will grow towards 1, the upper limit for $\tau$. Therefore, this will yield a unique valid solution for each valid value of $d$. 

\medskip

Now, assume that $\overline{p}^2\neq4p$. Then, using the quadratic formula, we can solve $T(\tau_\infty)=0$ for $\tau_{\infty}$:
\begin{equation*}
    \tau_{\infty}=\dfrac{\overline{p}^2(d+1)-4p\pm\sqrt{\left(4p-\overline{p}^2(d+1)\right)^2-4\left(\overline{p}^2-4p\right)\left(\overline{p}^2d-p\right)}}{2\left(\overline{p}^2-4p\right)}.
\end{equation*}

This gives two values for $\tau_{\infty}$. We know that $\tau\in\left[\frac12,1\right)$. Now check the end points in the $T(\tau_\infty)$ quadratic function to see if either root for $\tau_{\infty}$ is valid. 

\smallskip

In the case $\tau_{\infty}=1$, we have:

\begin{align*}
    T(1)&= \overline{p}^2-4p+4p-\overline{p}^2(d+1)+\overline{p}^2d-p = -p.
\end{align*}

Since $p\in(0,1)$, the original quadratic function must be negative for $\tau_{\infty}=1$, i.e. $T(1)<0$. 

\medskip

In the case $\tau_{\infty}=\frac12$, we have:

\begin{align*}
    T\left(\frac12\right) &= \frac14\left(\overline{p}^2-4p\right) + \frac12\left(4p-\overline{p}^2(d+1)\right) +\overline{p}^2d-p \\
    &= \overline{p}^2\left(d-\frac12\right). 
\end{align*}

Since $d>\frac12$, the original quadratic function must be positive for $\tau_{\infty}=\frac12$, i.e. $T(\frac12)>0$. Thus, exactly one root lies in the range for $\tau_{\infty}$. The appropriate root will depend on the coefficient of the $\tau_{\infty}^2$ in the quadratic equation, which we will call $\alpha \coloneqq \overline{p}^2-4p$, and more specifically, it will depend on the sign of $\alpha$. For our purposes, let $B=\overline{p}^2(d+1)-4p$ and $D=\left(4p-\overline{p}^2(d+1)\right)^2-4\left(\overline{p}^2-4p\right)\left(\overline{p}^2d-p\right)$, so that the two roots are $\tau_{\infty}=\frac{B\pm\sqrt{D}}{2\alpha}$. We are looking for the root such that $\frac12<\tau_\infty<1$.

If $\alpha>0$, then the quadratic function $T(\tau_\infty)$ will be concave up, and since $T(\frac{1}{2})>0$ and $T(1)<0$, then the valid root is the smaller root. In this case, that would be $\tau_{\infty}=\frac{B-\sqrt{D}}{2\alpha}$. 

If $\alpha<0$, then the quadratic function $T(\tau_\infty)$ will be concave down, and since $T(\frac{1}{2})>0$ and $T(1)<0$, then the valid root is the larger root. In this case, since $\alpha<0$, that would also be $\tau_{\infty}=\frac{B-\sqrt{D}}{2\alpha}$. Therefore, we know that the valid root will be: 

\begin{equation*}
    \tau_{\infty}=\dfrac{\overline{p}^2(d+1)-4p-\sqrt{\left(4p-\overline{p}^2(d+1)\right)^2-4\left(\overline{p}^2-4p\right)\left(\overline{p}^2d-p\right)}}{2\left(\overline{p}^2-4p\right)}.
\end{equation*}

\end{proof}

\section{Bounds on Expected Value}
\label{sec: bounds}

In this section, we establish bounds for the logarithm of the expected value of the $m^{\textrm{th}}$ Betti number. 

\begin{corollary}
For constants $d>\frac{1}{2}$ and $\tau$, where $N+1=dq$ and $\tau=\frac{m+1}{q}$, 
$$c_1\leq \frac{\log \mathbb{E}[\widetilde{\beta}_m]}{m} \leq c_2,$$
where $c_1=\left(2-\frac{1}{\tau}\right)\log\left(\frac{\overline{p}}{\tau}\right) + \log d + \left(\frac{1}{\tau}-1\right)\log p+o(1)$ and $c_2=c_1+\frac{1}{\tau}$

\end{corollary}

\begin{proof}
Here, the following bounds on binomial coefficients are used: 
$$\left(\frac{n}{k}\right)^{k} \leq \binom{n}{k} \leq \left(\frac{ne}{k}\right)^{k}.$$

Recall from Lemma~\ref{lemma: exval} that $\mathbb{E}[\widetilde{\beta}_m]=\binom{N}{m+1}\binom{m+1}{q-m-2}p^{q-m-2}\;\overline{p}^{2m+3-q}$. For now, let $R=p^{q-m-2}\;\overline{p}^{2m+3-q}$. Therefore, we have that 
\[
    \left(\frac{N}{m+1}\right)^{m+1}\left(\frac{m+1}{q-m-2}\right)^{q-m-2}R  \leq \mathbb{E}[\widetilde{\beta}_m] \leq \left(\frac{Ne}{m+1}\right)^{m+1}\left(\frac{(m+1)e}{q-m-2}\right)^{q-m-2}R,
\]
    
    which after simplification, leads to
    
\[
    N^{m+1} \left(\frac{p}{q}\right)^{q-m-2}\left(\frac{\overline{p}}{m+1}\right)^{2m+3-q} \leq \mathbb{E}[\widetilde{\beta}_m] \leq N^{m+1} \left(\frac{p}{q}\right)^{q-m-2}\left(\frac{\overline{p}}{m+1}\right)^{2m+3-q}e^{q-1}.
\]

Consider the lower bound, $\mathbb{E}[\widetilde{\beta}_m] \geq N^{m+1} \left(\frac{p}{q}\right)^{q-m-2}\left(\frac{\overline{p}}{m+1}\right)^{2m+3-q}$, and take the log of both sides, and note that the upper bound will differ only by $q-1$: 

\[
    \log \mathbb{E}[\widetilde{\beta}_m] \geq (m+1)\log N +  (q-m-2)\log\left(\frac{p}{q}\right) + (2m+3-q)\log\left(\frac{\overline{p}}{m+1}\right).
\]

Using the relation $N+1=dq$, we get:
\begin{align*}
    \log \mathbb{E}[\widetilde{\beta}_m] \geq& (m+1)\log(dq-1) +  (q-m-2)\log\left(\frac{p}{q}\right) + (2m+3-q)\log\left(\frac{\overline{p}}{m+1}\right) \\
    =& (2m+3-q)\log \left(\frac{\overline{p}q}{m+1}\right)+ (m+1)\log(d-1/q) +  (q-m-2)\log p. 
\end{align*}
Using the relation $\tau=\dfrac{m+1}{q}$, we get
\begin{align*}
    \log \mathbb{E}[\widetilde{\beta}_m] \geq & \left(2m+3-\frac{m+1}{\tau}\right)\log \left(\frac{\overline{p}}{\tau}\right) + (m+1)\log\left(d-\frac{\tau}{m+1}\right) \\
    &  + \left(\frac{m+1}{\tau}-m-2\right)\log p .
\end{align*}

Asymptotically, as $q\rightarrow\infty$, this is $m\left( \left(2-\frac{1}{\tau}\right)\log\left(\frac{\overline{p}}{\tau}\right) + \log d + \left(\frac{1}{\tau}-1\right)\log p\right)+o(m)$.

\end{proof}

Using the value for $\tau_\infty$ found in Theorem~\ref{thm: tinf}, we can plot the regions where the constants $c_1$ and $c_2$ change signs for $0\leq p\leq 1$ and $\frac12\leq d\leq 2$. Note that if $c_1<0$ and $c_2<0$, the expectation of the $m_{\textrm{peak}}$th Betti number is going to zero. If both $c_1>0$ and $c_2>0$, then the expectation of the $m_{\textrm{peak}}$th Betti number is growing. However, if $c_1<0$ and $c_2>0$, the bounds don't say much about the expectation of the $m_{\textrm{peak}}$th Betti number. 

\begin{center}
    \includegraphics[scale=.9]{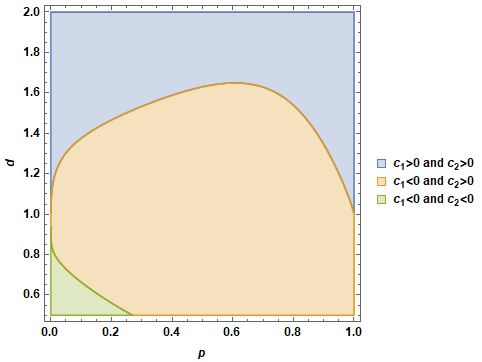}
    \captionof{figure}{Regions where the constants in the bounds for $\log \mathbb{E}[\widetilde{\beta}_{m_{\textrm{peak}}}]/m$ change sign.}
\end{center}

In Figure 1, the horizontal axis is the $p$ axis and the vertical axis is the $d$ axis. From this, it can be seen that there is a small (green) region for small $p$ and $d$ where both constants are negative, which implies that the expected Betti number peak is going to zero. It can also be seen that there is a blue region for larger values of $d$ where both constants are positive, which implies that the peak Betti number is growing.

\end{document}